\documentclass[11pt]{amsart}

\usepackage{enumerate, amsmath, amsthm, amsfonts, amssymb, xy,  mathrsfs, graphicx, paralist, lscape}
\usepackage[usenames, dvipsnames]{color}
\usepackage{fullpage, bbold}
\usepackage{pgfplots}
\pgfplotsset{width=7cm,compat=1.9}
\usepackage{tikz}
\usetikzlibrary{matrix,arrows}

\numberwithin{equation}{section}
\newtheorem{theorem}[equation]{Theorem}
\newtheorem*{thm}{Theorem}

\newtheorem{corollary}[equation]{Corollary}

\theoremstyle{definition}
\newtheorem{rmk}[equation]{Remark}

\newtheorem{rmks}[equation]{Remarks}

\newtheorem{eg}[equation]{Example}
\newenvironment{example}[1][]{\begin{eg}[#1] \pushQED{\qed}}{\popQED \end{eg}}
\newtheorem{defn}[equation]{Definition}
\newenvironment{definition}[1][]{\begin{defn}[#1]\pushQED{\qed}}{\popQED \end{defn}}
\newtheorem{altdefn}[equation]{Alternative Definition}

\newtheorem{ques}[equation]{Question}

\newtheorem{notn}[equation]{Notation}
\newenvironment{notation}[1][]{\begin{notn}[#1]}{\popQED \end{notn}}
\usepackage[colorlinks=true, pdfstartview=FitV, linkcolor=blue, citecolor=blue, urlcolor=blue]{hyperref}

\renewcommand{\phi}{\varphi}

\newcommand{\comment}[1]{}

\makeatletter
\def\Ddots{\mathinner{\mkern1mu\raise\p@
\vbox{\kern7\p@\hbox{.}}\mkern2mu
\raise4\p@\hbox{.}\mkern2mu\raise7\p@\hbox{.}\mkern1mu}}
\makeatother

\usetikzlibrary{calc}
\include{pipe-macros-stumpf}

\begin{document}

\title[Toric double determinantal varieties]{Toric double determinantal varieties}

\author[Alexander Blose]{Alexander Blose}
\author[Patricia Klein]{Patricia Klein}
\author[Owen McGrath]{Owen McGrath}
\author[Jackson Morris]{Jackson Morris}
\address{University of Kentucky, Department of Mathematics, Lexington, KY, USA.}
\email[Alexander Blose]{alexander.blose@uky.edu}
\address{University of Minnesota, School of Mathematics, Minneapolis, MN, USA.}
\email[Patricia Klein]{klein847@umn.edu}
\address{University of Kentucky, Department of Mathematics, Lexington, KY, USA.}
\email[Owen McGrath]{owen.mcgrath@uky.edu}
\address{University of Kentucky, Department of Mathematics, Lexington, KY, USA.}
\email[Jackson Morris]{jackson.morris@uky.edu}

\maketitle

\begin{abstract}
We examine Li's double determinantal varieties in the special case that they are toric.  We recover from the general double determinantal varieties case, via a more elementary argument, that they are irreducible and show that toric double determinantal varieties are smooth.  We use this framework to give a straighforward formula for their dimension.  Finally, we use the smallest nontrivial toric double determinantal variety to provide some empirical evidence concerning an open problem in local algebra.
\end{abstract}

\section{Introduction}  Double determinantal varieties, introduced by L. Li, are projective varieties that arise naturally from representation theory.  In particular, they are a special case of Nakajima quiver varieties, which were introduced to study Kac-Moody Lie algebras.  For background on Nakajima quiver varieties, see \cite{Nak} or \cite{Gin}, and for background on double determinantal varieties, see \cite{Li}.  A double determinantal variety is defined by the vanishing of minors of a fixed size in a concatenation of finitely many $m \times n$ matrices glued along their length $m$ edges together with the vanishing of minors of a possibly different size in a concatenation of the same matrices along their length $n$ edges (see Definition \ref{doubleDeterminantalVarieties}).  They were recently shown to be normal, irreducible, and arithmetically Cohen--Macaulay \cite{FK}.  

In this note, we study the special case of toric touble determinantal varieties, i.e. double determinantal varieties in which the minors in both the horizontal and vertical concatenations are size $2$.  For background on toric varieties, we refer the reader to \cite{CLS11}.  Our main results are the following:

\begin{thm}(Theorem \ref{thm: prime})
Toric double determinantal varieties are irreducible.
\end{thm}

We give an elementary proof, which recovers of the irreducible result of \cite[Corollary 4.7]{FK} in the toric case.

\begin{thm} (Theorem \ref{thm: smooth})
Toric double determinantal varieties are smooth projective varieties.
\end{thm}

The above result shows that toric double determinantal varieties are strictly better behaved than double determinantal varieties in the general case.  As a corollary (Corollary \ref{cor: dim(X)}), we are able to give a formula for the dimension of a toric double determinantal varieties that is much simpler than the formula \cite[Corollary 4.5]{FK} that applies to the not necessarily toric case.  

The structure of the paper is as follows: In Section \ref{sect: definitions}, we provide definitions and fix notation.  In Section \ref{sect: theorems}, we provide proofs of the main theorems.  In Section \ref{sect: testing}, we consider the smallest nontrivial toric double determinantal variety for empirical evidence concerning an open problem in local algebra.

\subsection*{Acknowledgements}  We are thankful to Courtney George, Chris Manon, Uwe Nagel for helpful conversations.  The first author was partially supported by the University of Kentucky's Summer Research \& Creativity Fellowship, and the third and fourth authors were partially supported by the Dr. J.C. Eaves Scholarship during the writing of this paper.  

\section{Preliminaries}\label{sect: definitions}

Throughout this paper, varieties will be defined over the perfect field $\kappa$.  For a matrix $M$ and a positive integer $a$, let $I_a(M)$ denote the ideal generated by the $a$-minors in $M$.  

\begin{definition}\label{doubleDeterminantalVarieties}
  Fix $r, m, n \geq 1$, and let $X_\ell = (x^\ell_{i,j})$ with $1 \leq \ell \leq r$ be $m \times n$ matrices of distinct indeterminates.  Let $R = \kappa[x^\ell_{i,j} \mid 1 \leq i \leq m, 1 \leq j \leq n, 1 \leq \ell \leq r]$ be the standard graded polynomial ring in the indeterminates that appear in the matrices $X_\ell$ over the perfect field $\kappa$. Let $H$ be the horizontal concatenation of these matrices, i.e., the $m \times rn$ matrix $H =
    \begin{bmatrix}
      X_1 \cdots X_r\\
    \end{bmatrix}$, and let $V$ be their vertical concatenation, i.e., the $rm \times n$ matrix

  \[V =
    \begin{bmatrix}
      X_{1} \\
      \vdots \\
      X_{r} \\
    \end{bmatrix}.  
  \]

  The ideal $I = I_a(H)+I_b(V)$ generated by the $a$-minors of $H$ together with the $b$-minors of $V$ is called a \emph{double determinantal ideal}, and the variety cut out by $I$ is called a \emph{double determinantal variety}.  Because $I$ homogeneous, we may think of the variety it defines as either an affine variety or a projective variety.  
    \end{definition}
  
  This paper concerns the special case of toric double determinantal varieties, which are double determinantal varieties that are also toric:  
  \begin{definition}
  An ideal in a polynomial ring is called \emph{toric} if it is prime and has a generating set in which every element is a binomial.  A variety is called \emph{toric} if it can be defined by a toric ideal.
  \end{definition}
 
When $a = b = 2$, it is clear that the double determinantal ideal $I$ is binomial.  Theorem \ref{thm: prime} shows that it is also prime so that the term toric double determinantal ideal is not a misnomer.  Notice that when $r = 1$, we are in the case of an ordinary determinantal variety, and so we will restrict our interest to the case $r \geq 2$.

\begin{notation}\label{notation} Throughout this paper, we will let $X_1 = (x^1_{ij}), \ldots, X_{r} = (x^{r}_{ij})$ be $r$ distinct $m \times n$ matrices of indeterminates for some integers $r \geq 2$ and $m, n \geq 2$.  We will always let $H = \begin{bmatrix} X_1 & \cdots & X_{r}\end{bmatrix}$ denote their $m \times rn$ horizontal concatenation and $ V = \begin{bmatrix}X_1\\
\vdots \\ X_{r}\end{bmatrix}$ their $rm \times n$ vertical concatenation. We will take $R = \kappa[x^\ell_{ij} \vert 1 \leq i \leq n, 1 \leq j \leq m, 1 \leq \ell \leq r]$ to be the polynomial ring in the entries of $H$ (or $V$) over the field $\kappa$, $I = I_2(H)+I_2(V)$, and $S = R/I$.  Notice that $S$ is the homogeneous coordinate ring of the toric double determinantal variety cut out by $I$, which we will call $\mathcal{X}$.  
\end{notation}

\section{Proofs of Main Results}\label{sect: theorems}

We begin by giving an elementary argument that every toric double determinantal ideal is prime.  Because a toric ideal must be both prime and binomial and the ideals in question are clearly binomial, this theorem establishes that the ideals we are calling toric double determinantal ideals are actually toric.  This result recovers \cite[Corollary 4.7]{FK} in the case of $2$-minors.

\begin{theorem}\label{thm: prime}
Using Notation \ref{notation}, $I$ is prime.  In particular, $S$ is a domain, and $\mathcal{X}$ is irreducible.
\end{theorem}
\begin{proof}

It is a standard fact about minors of a generic matrix that there will an isomorphism $A:=R/I_2(H) \cong k[s_i t_j \mid 1 \leq i \leq m, 1 \leq j \leq rn]$ by the map sending $x^\ell_{ij} \mapsto s_i{t_{\ell n+j}}$ for each $1 \leq i \leq n, 1 \leq j \leq m, 1 \leq \ell \leq r$.  Therefore, $S$ is the quotient of $A$ by the image of $I_2(V)$ in $A$.  The image of $I_2(V)$ in $A$ may be described as $I_2(\widetilde{V})$, where 

\begin{center}$\widetilde{V} = \begin{bmatrix}
	s_1t_1 & s_1t_2 & \dots & s_1t_{n} \\
	\vdots & \vdots & \ddots & \vdots \\
	s_mt_1 & s_mt_2 & \dots & s_mt_{n}\\
	s_1t_{n+1} & s_1t_{n+2} & \dots & s_1t_{2n} \\
	\vdots & \vdots & \ddots & \vdots \\
        \vdots & \vdots & \ddots & \vdots \\
	s_mt_{(r-1)n+1} & s_mt_{(r-1)n+2} & \dots & s_mt_{rn} \\
	\end{bmatrix}$.
\end{center}

Hence, it suffices to show that $I_2(\widetilde{V})$ is prime in $A$.  In order to prove this, we will show that $I_2(\widetilde{V})$ is the contraction of a prime ideal of the polynomial ring in the entries of $\widetilde{V}$.  In particular, let $U = k[s_i, t_j \mid 1 \leq i \leq m, 1 \leq j \leq rn]$ and $T = \begin{bmatrix}t_1 & t_2 & \dots & t_n \\
\vdots & \vdots & \ddots & \vdots\\
t_{(r-1)n+1} & t_{(r-1)n+2} & \dots & t_{rn}\end{bmatrix}$, and consider $P = I_2(T) \subseteq U$.  We know that $P$ is prime because it is the ideal of $2$-minors of a generic matrix.  We will show that $I_2(\widetilde{V}) = P \cap A$, which will establish that $I_2(\widetilde{V})$ is prime because the contraction of a prime ideal is always a prime ideal.  Because $I_2(\widetilde{V})$ and $P \cap A$ are both homogeneous ideals, it is enough to consider their homogeneous elements when checking that $I_2(\widetilde{V}) \subseteq P \cap A$ and that $P \cap A \subseteq I_2(\widetilde{V})$.

We first consider which homogemeous elements of $U$ (using the standard grading) are also elements of $A$.    For any monomial $\mu \in U$, define \[
\mbox{deg}_s(\mu): = \max\{\alpha_1+\ldots+\alpha_m \mid s_1^{\alpha_1} \cdots s_m^{\alpha_m} \mbox{ divides } \mu\}
\]  and \[
\mbox{deg}_t(\mu): = \max \{\alpha_1+\ldots+\alpha_{rn} \mid t_1^{\alpha_1} \cdots t_{rn}^{\alpha_{rn}} \mbox{ divides } \mu\}.
\]  We observe that a monomial $\mu$ of $U$ is an element of $A$ if and only if $\mbox{deg}_s(\mu) = \mbox{deg}_t(\mu)$ and, more generally, that a homogeneous polynomial $h$ of $U$ is also an element of $A$ if and only if $\mbox{deg}_s(\mu) = \mbox{deg}_t(\mu)$ for every monomial $\mu$ occurring with nonzero coefficient in the standard expression of $h$.

Fix an arbitrary $2$-minor of the ideal $I_2(\widetilde{V})$.  These, the natural generators of $I_2(\widetilde{V})$, are of the form $(s_it_j)(s_at_b)-(s_it_{b-cn})(s_at_{j+cn})$ for some integers $i,j,a,b,c$ with $1 \leq i,a \leq m$,  and $1 \leq j,b, b-cn, j+cn \leq rn$.  By factoring out $s_is_a$, we see each natural generator of $I_2(\widetilde{V})$ expressed as $s_is_a(t_jt_b- t_{b-cn}t_{j+cn})$ for integers $i,j,a,b,c$ with $1 \leq i,a \leq m$, and $1 \leq j,b, b-cn, j+cn \leq rn$.  But every $t_jt_b- t_{b-cn}t_{j+cn}$ for integers $j,b,c$ with $1 \leq j,b, b-cn, j+cn \leq rn$ is either $0$ or one of the natural generators of $I_2(T)$.  Hence, every generator of $I_2(\widetilde{V})$ is an element of $I_2(T) \cap A$, which is to say $I_2(\widetilde{V}) \subseteq I_2(T) \cap A$.  

Conversely, fix an arbitrary homogeneous element $f \in I_2(\widetilde{V}) \cap S$.  Express $f = \mu_1 \delta_1 + \cdots+\mu_p \delta_p$, where the $\delta_i$ are (not necessarily distinct) natural generators of $I_2(\widetilde{V})$, the $\mu_i$ are monomials in $U$, and $p$ is a positive integer.    Hence, each term $\alpha_w \delta_w \in A$ if and only if $\deg_s(\alpha_w) = \deg_t(\alpha_w)+2$ because each $\deg_t(\delta_w) = 2$ while $\deg_s(\delta_w) = 0$.  Now $A$ is a homogeneous subring of $U$ under the $\mathbb{N} \times \mathbb{N}$ bigrading given by $(\deg_s, \deg_t)$, and so we may assume that each summand $\alpha_w\delta_w$ of $f$ is an element of $A$, which is to say that for each $1 \leq w \leq p$, we have $\deg_s(\alpha_w) = \deg_t(\alpha_w)+2$.  But then there exists some $1 \leq i,a \leq m$ such that $s_is_a$ divides $\alpha_w$ and $\dfrac{\alpha_w}{s_is_a}((s_is_a)\delta_w)$ is the product of a monomial $\dfrac{\alpha_w}{s_is_a}$ with $\mbox{deg}_s(\dfrac{\alpha_w}{s_is_a}) = \mbox{deg}_t(\dfrac{\alpha_w}{s_is_a})$ and $(s_is_a)\delta_w$.  That is to say that $\dfrac{\alpha_w}{s_is_a}((s_is_a)\delta_w)$ is the product of an element of $A$ and one of the natural generators of $I_2(\widetilde{V})$.  Hence, $I_2(\widetilde{V}) \supseteq I_2(T) \cap A$.

We have now seen that $I$ is the contraction of the prime ideal $I_2(\widetilde{V})$ and so is itself prime, as desired.

\end{proof}

We next describe a way in which toric double determinantal varieties are better behaved than the general case of double determinantal varieties.  Although every double determinantal variety is Cohen--Macaulay, they are not in general smooth as projective varieties, i.e. their homogeneous coordinate rings are not necessarily regular at the localization at every non-maximal homogeneous prime.  

\begin{theorem}\label{thm: smooth}
Every toric double determinantal variety is smooth as a projective variety.
\end{theorem}

\begin{proof}
Using Notation \ref{notation}, we must show that $S$ is regular at the localization at every homogeneous prime except possibly the irrelevant ideal, i.e. the homogeneous maximal ideal $\mu$.

Because no prime ideal $P \neq \mu$ of $S$ contains every indeterminate $x^\ell_{ij}$, every localization $S_P$ of $S$ at any homogeneous prime ideal $P \neq m$ may be obtained as the localization of some $S[1/x^\ell_{ij}]$ at the image of $P$ in $S[1/x^\ell_{ij}]$.  Now because the localization of a regular ring is again regular, it suffices to show that every ring of the form $S[1/x^\ell_{ij}]$ is regular.  

Fix one indeterminate $x^{\ell'}_{i'j'}$.  Let $Y_1$ denote the subset of the $x^\ell_{ij}$ algebra generators of $R$ satisfying at least two of the following conditions: (1) $i=i'$, (2) $j=j'$, (3) $\ell = \ell'$.  Informally, we are choosing among the algebra generators of $R$ the variables originating from any matrix $X^\ell$ that share both a row and column index with our chosen generator $x^{\ell'}_{i'j'}$ together with those algebra generators originating from the same matrix $X^{\ell'}$ as $x^{\ell'}_{i'j'}$ that share either a row or column index with $x^{\ell'}_{i'j'}$.  These are exactly the algebra generators $x^\ell_{ij}$ of $R$ so that $x^\ell_{ij} \cdot x^{\ell'}_{i'j'}$ does not appear as a term of any of the natural generators of $I$.  For each $x^\ell_{ij} \in Y_1$, designate a new variable $z^\ell_{ij}$ and for the ring $T=\kappa[z^\ell_{ij} \mid x^\ell_{ij} \in Y_1]$.  We claim that $S[1/x^{\ell'}_{i'j'}] \cong T[1/z^{\ell'}_{i'j'}]$.  

To show the promised isomorphism, we will give explicit homomorphisms $\phi: S[1/x^{\ell'}_{i'j'}] \rightarrow T$ and $\psi: T \rightarrow S[1/x^{\ell'}_{i'j'}]$ so that $\phi \circ \psi  = \textbf{1}_{T[1/z^{\ell'}_{i'j'}]}$ and $\psi \circ \phi = \textbf{1}_{S[1/x^{\ell'}_{i'j'}]}$.  In order to do this, we consider three subsets of the free algebra generators of $R$, the first of which is $Y_1$.  The next subset is $Y_2$, which will denote the subset of the $x^\ell_{ij}$ such that exactly one of the following conditions is met: (1) $i=i'$, (2) $j=j'$, (3) $\ell = \ell'$.  Finally, take $Y_3$ to be the subset of the the subset of the $x^\ell_{ij}$ such that $i \neq i'$, $j \neq j'$, and $\ell \neq \ell'$.  By construction, the sets $Y_1$, $Y_2$, and $Y_3$ are pairwise disjoint.  

We explain an overview of the reason for this subdivision before we continue with the formal argument.  Notice that the elements $x^\ell_{ij}$ of $Y_2$ are those such that $x^\ell_{ij} \cdot x^{\ell'}_{i'j'}$ is a term of exactly one of the natural generators of $I$ and that the variables other than $x^\ell_{ij}$ involved in that generator are all elements of $Y_1$.  We will use that generator to solve for elements of $Y_2$ in terms of the elements of $Y_1$.  We may then use a generator of $I$ involving each element of $Y_3$ to solve for that element in terms of elements of $Y_2$ and $Y_1$, which we can then simplify to an expression in terms of the elements of $Y_1$.  

We define $\phi:R[1/x^{\ell'}_{i'j'}] \rightarrow T[1/z^{\ell'}_{i'j'}]$ by specifying the image of the free algebra generators $x^{\ell}_{ij}$ of $R$ over $\kappa$, extending linearly, and showing that $\varphi(IR[1/x^{\ell'}_{i'j'}]) = 0$, in which case we may think of $\phi$ as a map from $S[1/x^{\ell'}_{i'j'}]$ to $T[1/z^{\ell'}_{i'j'}]$.  It will be immediate that $\varphi(\kappa[x^{\ell'}_{i'j'}, 1/x^{\ell'}_{i'j'}]) = \kappa[z^{\ell'}_{i'j'}, 1/z^{\ell'}_{i'j'}]$.  For any $s \in S$ and $t \in T$, we may write simply $s$ or $t$ for $s/1 \in S[1/x^{\ell'}_{i'j'}]$ or $t/1 \in T[1/z^{\ell'}_{i'j'}]$.  We will group cases by notational convenience.  Let \[
\phi(x_{ij}^\ell) =  \begin{cases} 
      z^\ell_{ij} & x^\ell_{ij} \in Y_1 \\
      \dfrac{z^{\ell'}_{ij'}z^{\ell'}_{i'j}}{z^{\ell'}_{i'j'}} & x^\ell_{ij} \in Y_2 \mbox{ and } \ell = \ell' \\
      \dfrac{z^{\ell'}_{i'j}z^\ell_{i'j'}}{z^{\ell'}_{i'j'}} & x^\ell_{ij} \in Y_2 \mbox{ and }i = i' \\
      \dfrac{z^{\ell'}_{ij'}z^\ell_{i'j'}}{z^{\ell'}_{i'j'}} & x^\ell_{ij} \in Y_2 \mbox{ and }j = j' \\
      \dfrac{z^{\ell'}_{i'j}z^{\ell'}_{ij'}z^{\ell}_{i'j'}}{(z^{\ell'}_{i'j'})^2} & x^\ell_{ij} \in Y_3.
   \end{cases}
\]  

We first claim that $\phi(IR[1/x^{\ell'}_{i'j'}]) = 0$.  It is enough to show that $\phi(\delta) = 0$ for every generator $\delta$ of $I$.  If $\delta$ involves $x^{\ell'}_{i'j'}$, then there are two cases: one of the summands of $\delta$ is a term $x^{\ell'}_{i'j'} \cdot x^{\ell}_{ij}$ with either $x^{\ell}_{ij} \in Y_2$ or $x^\ell_{ij} \in Y_3$.  If $x^\ell_{ij} \in Y_2$ and $\ell = \ell'$ (respectively, $i=i'$ or $j=j'$), then $\delta$ has the form 
$x^{\ell'}_{i'j'}x^{\ell'}_{ij}-x^{\ell'}_{ij'}x^{\ell'}_{i'j}$ (respectively, $x^{\ell'}_{i'j'}x^\ell_{i'j}-x^{\ell'}_{i'j}x^{\ell}_{i'j'}$ or $x^{\ell'}_{i'j'}x^\ell_{ij'}-x^{\ell'}_{ij'}x^{\ell}_{i'j'}$), and $x^{\ell'}_{ij'}, x^{\ell'}_{i'j} \in Y_1$ (respectively, $x^{\ell'}_{i'j}x^{\ell}_{i'j'} \in Y_1$ or $x^{\ell'}_{ij'}x^{\ell}_{i'j'} \in Y_1$).  Hence, $\phi(\delta) = z^{\ell'}_{i'j'}\frac{z^{\ell'}_{ij'}z^{\ell'}_{i'j}}{z^{\ell'}_{i'j'}}-z^{\ell'}_{ij'}z^{\ell'}_{i'j}$ (respectively, $z^{\ell'}_{i'j'}\frac{z^{\ell'}_{i'j}z^\ell_{i'j'}}{z^{\ell'}_{i'j'}}-z^{\ell'}_{i'j}z^\ell_{i'j'}$ or $z^{\ell'}_{i'j'}\frac{z^{\ell'}_{ij'}z^\ell_{i'j'}}{z^{\ell'}_{i'j'}}-z^{\ell'}_{ij'}z^\ell_{i'j'}$), which is plainly $0$.    Similarly, if $x^\ell_{ij} \in Y_3$ and $\delta$ is a natural generator of $I_2(H)$ (respectively, of $I_2(V)$), then $\delta$ has the form $x^{\ell'}_{i'j'} x^\ell_{ij}-x^{\ell}_{i'j}x^{\ell'}_{ij'}$ with $x^\ell_{i'j} \in Y_2$ and $x^{\ell'}_{ij'} \in Y_1$ (respectively, $x^{\ell'}_{i'j'} x^\ell_{ij}-x^{\ell}_{ij'}x^{\ell'}_{i'j}$ with $x^{\ell}_{ij'} \in Y_2$ and $x^{\ell'}_{i'j} \in Y_1$).   Then $\varphi(\delta) =  z^{\ell'}_{i'j'} \frac{z^{\ell'}_{i'j}z^{\ell'}_{ij'}z^{\ell}_{i'j'}}{(z^{\ell'}_{i'j'})^2}-\frac{z^{\ell'}_{i'j}z^\ell_{i'j'}}{z^{\ell'}_{i'j'}} z^{\ell'}_{ij'}$ (respectively, $z^{\ell'}_{i'j'} \frac{z^{\ell'}_{i'j}z^{\ell'}_{ij'}z^{\ell}_{i'j'}}{(z^{\ell'}_{i'j'})^2}-\frac{z^{\ell'}_{ij'}z^\ell_{i'j'}}{z^{\ell'}_{i'j'}}z^{\ell'}_{i'j}$), which is plainly $0$.  Hence, $\phi$ induces a map $S[1/x^{\ell'}_{i'j'}] \rightarrow T[1/z^{\ell'}_{i'j'}]$, which we will also refer to as $\varphi$.

We define the map $\psi:T[1/z^{\ell'}_{i'j'}] \rightarrow S[1/x^\ell_{ij}]$ by specifying $\psi(z^\ell_{ij}) = x^{\ell}_{ij}$ for every $i$, $j$, and $\ell$ and extending linearly.  It is clear that $\varphi \circ \psi = \mathbb{1}_{T[1/z^{\ell'}_{i'j'}]}$ and that $(\psi \circ \varphi) (x^\ell_{ij}) = x^\ell_{ij}$ for all $x^\ell_{ij} \in Y_1$.  If $x^\ell_{ij} \in Y_2$ and $\ell = \ell'$, then \[
(\psi \circ \varphi )(x^\ell_{ij}) =(\psi \circ \varphi )(x^{\ell'}_{ij}) = \psi\left(\frac{x^{\ell'}_{ij'}x^{\ell'}_{i'j}}{x^{\ell'}_{i'j'}}\right) = \frac{x^{\ell'}_{ij'}x^{\ell'}_{i'j}}{x^{\ell'}_{i'j'}}= x^{\ell'}_{ij} \in S[1/x^{\ell'}_{i'j'}]
\] because $x^{\ell'}_{ij'}x^{\ell'}_{i'j}-x^{\ell'}_{i'j'}x^{\ell'}_{ij} \in I$.  A similar argument holds in the case of $i=i'$ (respectively, $j=j'$) recalling that the definition of $Y_2$ forces $x^{\ell'}_{i'j}x^{\ell}_{i'j'}-x^{\ell'}_{i'j'}x^{\ell}_{i'j} \in I_2(V)$  (respectively, $x^{\ell'}_{ij'}x^{\ell}_{i'j'}-x^{\ell'}_{i'j'}x^{\ell}_{ij'} \in I_2(H)$) whenever $x^\ell_{i'j} \in Y_2$  (respectively, $x^\ell_{ij'} \in Y_2$).  Finally, if $x^\ell_{ij} \in Y_3$, then $(\psi \circ \varphi) (x^\ell_{ij}) = \psi\left(\frac{x^{\ell'}_{i'j}x^{\ell'}_{ij'}x^{\ell}_{i'j'}}{(x^{\ell'}_{i'j'})^2}\right) =\frac{x^{\ell'}_{i'j}x^{\ell'}_{ij'}x^{\ell}_{i'j'}}{(x^{\ell'}_{i'j'})^2} = x^\ell_{ij} \in S[1/x^\ell_{ij}]$ because \[
x^\ell_{ij}(x^{\ell'}_{i'j'})^2-x^{\ell'}_{i'j}x^{\ell'}_{ij'}x^{\ell}_{i'j'} = x^{\ell'}_{i'j'}(x^{\ell'}_{i'j'}x^\ell_{ij}-x^\ell_{i'j}x^{\ell'}_{ij'})+x^{\ell'}_{ij'}(x^{\ell'}_{i'j'}x^{\ell}_{i'j}-x^\ell_{i'j'}x^{\ell'}_{i'j}),
\] where $x^{\ell'}_{i'j'}x^\ell_{ij}-x^\ell_{i'j}x^{\ell'}_{ij'}$ and $x^{\ell'}_{i'j'}x^{\ell}_{i'j}-x^\ell_{i'j'}x^{\ell'}_{i'j} $ can be seen to be natural generators of $I_2(H)$ and $I_2(V)$, respectively, using the fact that $x^\ell_{ij} \in Y_3$ requires $\ell \neq \ell'$, $i \neq i'$, and $j \neq j'$.

Hence, $\varphi$ and $\psi$ are inverses, $S[1/x^{\ell'}_{i'j'}] \cong T[1/z^{\ell'}_{i'j'}]$, $S[1/x^{\ell'}_{i'j'}]$ is regular because $T[1/z^{\ell'}_{i'j'}]$ is, and $S$ is regular at the localization at every homogeneous prime except possibly the irrelevant ideal, which is to say that $\mathcal{X}$ is a smooth projective variety.  

\end{proof}

\begin{example}
With notation as in the theorem above, if $r = 2$, $m = n = 3$, and $x^{\ell'}_{i'j'} = x^1_{32}$ then the partition of the variables $x^{\ell}_{ij}$ into the sets $Y_1$, $Y_2$, and $Y_3$ may be seen below with the elements of $Y_1$ placed in orange boxes, the elements of $Y_2$ placed in black boxes, and the elements of $Y_3$ placed in green boxes.  For notational convenience, we will use $x_{ij}$ to denote the $x^1_{ij}$ and $y_{ij}$ to denote the $x^2_{ij}$.  This small example gives an adequate sense of the relations catalogued more precisely in the proof above. 

\begin{minipage}{0.5\textwidth}
  \begin{equation*}
    H = \begin{tikzpicture}[baseline=(current  bounding  box.center)]
      \matrix [matrix of math nodes,left delimiter=(,right delimiter=)] (m)
      {
        x_{11} & x_{12} & x_{13} & y_{11} & y_{12} & y_{13} \\
        x_{21} & x_{22} & x_{23} & y_{21} & y_{22} & y_{23} \\
        x_{31} & x_{32} & x_{33} & y_{31} & y_{32} & y_{33}\\
      };
      \draw[black, thick]
      ([xshift = 1pt, yshift = -1pt] m-1-1.north west) --
      ([xshift = 1pt, yshift = 1pt]m-1-1.south west) --
      ([xshift = -1pt, yshift = 1pt]m-1-1.south east) --
      ([xshift = -1pt, yshift = -1pt]m-1-1.north east) --
      ([xshift = 1pt, yshift = -1pt]m-1-1.north west);
             \draw[black, thick]
      ([xshift = 1pt, yshift = -1pt]m-2-1.north west) --
      ([xshift = 1pt, yshift = 1pt]m-2-1.south west) --
      ([xshift = -1pt, yshift = 1pt]m-2-1.south east) --
      ([xshift = -1pt, yshift = -1pt]m-2-1.north east) --
      ([xshift = 1pt, yshift = -1pt]m-2-1.north west);
            \draw[black, thick]
      ([xshift = 1pt, yshift = -1pt]m-1-3.north west) --
      ([xshift = 1pt, yshift = 1pt]m-1-3.south west) --
      ([xshift = -1pt, yshift = 1pt]m-1-3.south east) --
      ([xshift = -1pt, yshift = -1pt]m-1-3.north east) --
      ([xshift = 1pt, yshift = -1pt]m-1-3.north west);
              \draw[black, thick]
      ([xshift = 1pt, yshift = -1pt]m-2-3.north west) --
      ([xshift = 1pt, yshift = 1pt]m-2-3.south west) --
      ([xshift = -1pt, yshift = 1pt]m-2-3.south east) --
      ([xshift = -1pt, yshift = -1pt]m-2-3.north east) --
      ([xshift = 1pt, yshift = -1pt]m-2-3.north west);
                \draw[black, thick]
      ([xshift = 1pt, yshift = -1pt]m-3-4.north west) --
      ([xshift = 1pt, yshift = 1pt]m-3-4.south west) --
      ([xshift = -1pt, yshift = 1pt]m-3-4.south east) --
      ([xshift = -1pt, yshift = -1pt]m-3-4.north east) --
      ([xshift = 1pt, yshift = -1pt]m-3-4.north west);
               \draw[black, thick]
      ([xshift = 1pt, yshift = -1pt]m-3-6.north west) --
      ([xshift = 1pt, yshift = 1pt]m-3-6.south west) --
      ([xshift = -1pt, yshift = 1pt]m-3-6.south east) --
      ([xshift = -1pt, yshift = -1pt]m-3-6.north east) --
      ([xshift = 1pt, yshift = -1pt]m-3-6.north west);
                \draw[black, thick]
      ([xshift = 1pt, yshift = -1pt]m-2-5.north west) --
      ([xshift = 1pt, yshift = 1pt]m-2-5.south west) --
      ([xshift = -1pt, yshift = 1pt]m-2-5.south east) --
      ([xshift = -1pt, yshift = -1pt]m-2-5.north east) --
      ([xshift = 1pt, yshift = -1pt]m-2-5.north west);
                   \draw[black, thick]
      ([xshift = 1pt, yshift = -1pt]m-1-5.north west) --
      ([xshift = 1pt, yshift = 1pt]m-1-5.south west) --
      ([xshift = -1pt, yshift = 1pt]m-1-5.south east) --
      ([xshift = -1pt, yshift = -1pt]m-1-5.north east) --
      ([xshift = 1pt, yshift = -1pt]m-1-5.north west);
      \draw[orange, thick]
        ([xshift = 1pt, yshift = -1pt]m-3-1.north west) --
      ([xshift = 1pt, yshift = 1pt]m-3-1.south west) --
      ([xshift = -1pt, yshift = 1pt]m-3-1.south east) --
      ([xshift = -1pt, yshift = -1pt]m-3-1.north east) --
      ([xshift = 1pt, yshift = -1pt]m-3-1.north west);
         \draw[orange, thick]
      ([xshift = 1pt, yshift = -1pt]m-3-2.north west) --
      ([xshift = 1pt, yshift = 1pt]m-3-2.south west) --
      ([xshift = -1pt, yshift = 1pt]m-3-2.south east) --
      ([xshift = -1pt, yshift = -1pt]m-3-2.north east) --
      ([xshift = 1pt, yshift = -1pt]m-3-2.north west);
           \draw[orange, thick]
      ([xshift = 1pt, yshift = -1pt]m-3-3.north west) --
      ([xshift = 1pt, yshift = 1pt]m-3-3.south west) --
      ([xshift = -1pt, yshift = 1pt]m-3-3.south east) --
      ([xshift = -1pt, yshift = -1pt]m-3-3.north east) --
      ([xshift = 1pt, yshift = -1pt]m-3-3.north west);
            \draw[orange, thick]
      ([xshift = 1pt, yshift = -1pt]m-2-2.north west) --
      ([xshift = 1pt, yshift = 1pt]m-2-2.south west) --
      ([xshift = -1pt, yshift = 1pt]m-2-2.south east) --
      ([xshift = -1pt, yshift = -1pt]m-2-2.north east) --
      ([xshift = 1pt, yshift = -1pt]m-2-2.north west);
          \draw[orange, thick]
      ([xshift = 1pt, yshift = -1pt] m-1-2.north west) --
      ([xshift = 1pt, yshift = 1pt]m-1-2.south west) --
      ([xshift = -1pt, yshift = 1pt]m-1-2.south east) --
      ([xshift = -1pt, yshift = -1pt]m-1-2.north east) --
      ([xshift = 1pt, yshift = -1pt]m-1-2.north west);
             \draw[orange, thick]
      ([xshift = 1pt, yshift = -1pt]m-3-5.north west) --
      ([xshift = 1pt, yshift = 1pt]m-3-5.south west) --
      ([xshift = -1pt, yshift = 1pt]m-3-5.south east) --
      ([xshift = -1pt, yshift = -1pt]m-3-5.north east) --
      ([xshift = 1pt, yshift = -1pt]m-3-5.north west);
                 \draw[green,thick]
      ([xshift = 1pt, yshift = -1pt]m-1-4.north west) --
      ([xshift = 1pt, yshift = 1pt]m-1-4.south west) --
      ([xshift = -1pt, yshift = 1pt]m-1-4.south east) --
      ([xshift = -1pt, yshift = -1pt]m-1-4.north east) --
      ([xshift = 1pt, yshift = -1pt]m-1-4.north west);
            \draw[green,thick]
      ([xshift = 1pt, yshift = -1pt]m-2-4.north west) --
      ([xshift = 1pt, yshift = 1pt]m-2-4.south west) --
      ([xshift = -1pt, yshift = 1pt]m-2-4.south east) --
      ([xshift = -1pt, yshift = -1pt]m-2-4.north east) --
      ([xshift = 1pt, yshift = -1pt]m-2-4.north west);
                       \draw[green,thick]
      ([xshift = 1pt, yshift = -1pt]m-1-6.north west) --
      ([xshift = 1pt, yshift = 1pt]m-1-6.south west) --
      ([xshift = -1pt, yshift = 1pt]m-1-6.south east) --
      ([xshift = -1pt, yshift = -1pt]m-1-6.north east) --
      ([xshift = 1pt, yshift = -1pt]m-1-6.north west);
            \draw[green,thick]
      ([xshift = 1pt, yshift = -1pt]m-2-6.north west) --
      ([xshift = 1pt, yshift = 1pt]m-2-6.south west) --
      ([xshift = -1pt, yshift = 1pt]m-2-6.south east) --
      ([xshift = -1pt, yshift = -1pt]m-2-6.north east) --
      ([xshift = 1pt, yshift = -1pt]m-2-6.north west);
    \end{tikzpicture}
  \end{equation*}
  \end{minipage}
  \begin{minipage}{0.5\textwidth}
  \begin{equation*}
  V = \begin{tikzpicture}[baseline=(current  bounding  box.center)]
      \matrix [matrix of math nodes,left delimiter=(,right delimiter=)] (m)
      {
        x_{11} & x_{12} & x_{13}  \\
        x_{21} & x_{22} & x_{23} \\
        x_{31} & x_{32} & x_{33} \\
        y_{11} & y_{12} & y_{13}\\
        y_{21} & y_{22} & y_{23} \\
        y_{31} & y_{32} & y_{33}\\
      };
         \draw[black, thick]
      ([xshift = 1pt, yshift = -1pt] m-1-1.north west) --
      ([xshift = 1pt, yshift = 1pt]m-1-1.south west) --
      ([xshift = -1pt, yshift = 1pt]m-1-1.south east) --
      ([xshift = -1pt, yshift = -1pt]m-1-1.north east) --
      ([xshift = 1pt, yshift = -1pt]m-1-1.north west);
             \draw[black, thick]
      ([xshift = 1pt, yshift = -1pt]m-2-1.north west) --
      ([xshift = 1pt, yshift = 1pt]m-2-1.south west) --
      ([xshift = -1pt, yshift = 1pt]m-2-1.south east) --
      ([xshift = -1pt, yshift = -1pt]m-2-1.north east) --
      ([xshift = 1pt, yshift = -1pt]m-2-1.north west);
            \draw[black, thick]
      ([xshift = 1pt, yshift = -1pt]m-1-3.north west) --
      ([xshift = 1pt, yshift = 1pt]m-1-3.south west) --
      ([xshift = -1pt, yshift = 1pt]m-1-3.south east) --
      ([xshift = -1pt, yshift = -1pt]m-1-3.north east) --
      ([xshift = 1pt, yshift = -1pt]m-1-3.north west);
              \draw[black, thick]
      ([xshift = 1pt, yshift = -1pt]m-2-3.north west) --
      ([xshift = 1pt, yshift = 1pt]m-2-3.south west) --
      ([xshift = -1pt, yshift = 1pt]m-2-3.south east) --
      ([xshift = -1pt, yshift = -1pt]m-2-3.north east) --
      ([xshift = 1pt, yshift = -1pt]m-2-3.north west);
                \draw[black, thick]
      ([xshift = 1pt, yshift = -1pt]m-6-1.north west) --
      ([xshift = 1pt, yshift = 1pt]m-6-1.south west) --
      ([xshift = -1pt, yshift = 1pt]m-6-1.south east) --
      ([xshift = -1pt, yshift = -1pt]m-6-1.north east) --
      ([xshift = 1pt, yshift = -1pt]m-6-1.north west);
               \draw[black, thick]
      ([xshift = 1pt, yshift = -1pt]m-6-3.north west) --
      ([xshift = 1pt, yshift = 1pt]m-6-3.south west) --
      ([xshift = -1pt, yshift = 1pt]m-6-3.south east) --
      ([xshift = -1pt, yshift = -1pt]m-6-3.north east) --
      ([xshift = 1pt, yshift = -1pt]m-6-3.north west);
                \draw[black, thick]
      ([xshift = 1pt, yshift = -1pt]m-5-2.north west) --
      ([xshift = 1pt, yshift = 1pt]m-5-2.south west) --
      ([xshift = -1pt, yshift = 1pt]m-5-2.south east) --
      ([xshift = -1pt, yshift = -1pt]m-5-2.north east) --
      ([xshift = 1pt, yshift = -1pt]m-5-2.north west);
                   \draw[black, thick]
      ([xshift = 1pt, yshift = -1pt]m-4-2.north west) --
      ([xshift = 1pt, yshift = 1pt]m-4-2.south west) --
      ([xshift = -1pt, yshift = 1pt]m-4-2.south east) --
      ([xshift = -1pt, yshift = -1pt]m-4-2.north east) --
      ([xshift = 1pt, yshift = -1pt]m-4-2.north west);
      \draw[orange, thick]
        ([xshift = 1pt, yshift = -1pt]m-3-1.north west) --
      ([xshift = 1pt, yshift = 1pt]m-3-1.south west) --
      ([xshift = -1pt, yshift = 1pt]m-3-1.south east) --
      ([xshift = -1pt, yshift = -1pt]m-3-1.north east) --
      ([xshift = 1pt, yshift = -1pt]m-3-1.north west);
         \draw[orange, thick]
      ([xshift = 1pt, yshift = -1pt]m-3-2.north west) --
      ([xshift = 1pt, yshift = 1pt]m-3-2.south west) --
      ([xshift = -1pt, yshift = 1pt]m-3-2.south east) --
      ([xshift = -1pt, yshift = -1pt]m-3-2.north east) --
      ([xshift = 1pt, yshift = -1pt]m-3-2.north west);
           \draw[orange, thick]
      ([xshift = 1pt, yshift = -1pt]m-3-3.north west) --
      ([xshift = 1pt, yshift = 1pt]m-3-3.south west) --
      ([xshift = -1pt, yshift = 1pt]m-3-3.south east) --
      ([xshift = -1pt, yshift = -1pt]m-3-3.north east) --
      ([xshift = 1pt, yshift = -1pt]m-3-3.north west);
            \draw[orange, thick]
      ([xshift = 1pt, yshift = -1pt]m-2-2.north west) --
      ([xshift = 1pt, yshift = 1pt]m-2-2.south west) --
      ([xshift = -1pt, yshift = 1pt]m-2-2.south east) --
      ([xshift = -1pt, yshift = -1pt]m-2-2.north east) --
      ([xshift = 1pt, yshift = -1pt]m-2-2.north west);
          \draw[orange, thick]
      ([xshift = 1pt, yshift = -1pt] m-1-2.north west) --
      ([xshift = 1pt, yshift = 1pt]m-1-2.south west) --
      ([xshift = -1pt, yshift = 1pt]m-1-2.south east) --
      ([xshift = -1pt, yshift = -1pt]m-1-2.north east) --
      ([xshift = 1pt, yshift = -1pt]m-1-2.north west);
             \draw[orange, thick]
      ([xshift = 1pt, yshift = -1pt]m-6-2.north west) --
      ([xshift = 1pt, yshift = 1pt]m-6-2.south west) --
      ([xshift = -1pt, yshift = 1pt]m-6-2.south east) --
      ([xshift = -1pt, yshift = -1pt]m-6-2.north east) --
      ([xshift = 1pt, yshift = -1pt]m-6-2.north west);
                 \draw[green,thick]
      ([xshift = 1pt, yshift = -1pt]m-4-1.north west) --
      ([xshift = 1pt, yshift = 1pt]m-4-1.south west) --
      ([xshift = -1pt, yshift = 1pt]m-4-1.south east) --
      ([xshift = -1pt, yshift = -1pt]m-4-1.north east) --
      ([xshift = 1pt, yshift = -1pt]m-4-1.north west);
            \draw[green,thick]
      ([xshift = 1pt, yshift = -1pt]m-5-1.north west) --
      ([xshift = 1pt, yshift = 1pt]m-5-1.south west) --
      ([xshift = -1pt, yshift = 1pt]m-5-1.south east) --
      ([xshift = -1pt, yshift = -1pt]m-5-1.north east) --
      ([xshift = 1pt, yshift = -1pt]m-5-1.north west);
                       \draw[green,thick]
      ([xshift = 1pt, yshift = -1pt]m-4-3.north west) --
      ([xshift = 1pt, yshift = 1pt]m-4-3.south west) --
      ([xshift = -1pt, yshift = 1pt]m-4-3.south east) --
      ([xshift = -1pt, yshift = -1pt]m-4-3.north east) --
      ([xshift = 1pt, yshift = -1pt]m-4-3.north west);
            \draw[green,thick]
      ([xshift = 1pt, yshift = -1pt]m-5-3.north west) --
      ([xshift = 1pt, yshift = 1pt]m-5-3.south west) --
      ([xshift = -1pt, yshift = 1pt]m-5-3.south east) --
      ([xshift = -1pt, yshift = -1pt]m-5-3.north east) --
      ([xshift = 1pt, yshift = -1pt]m-5-3.north west);
          \end{tikzpicture}
  \end{equation*}
  \end{minipage}

\end{example}

\begin{corollary}\label{cor: dim(X)}
If $\mathcal{X}$ is a toric double determinantal variety on $r$ matrices of size $m \times n$, then the dimension of the projective variety $\mathcal{X}$ is $m+n+r-3$.
\end{corollary}

\begin{proof}
With notation as in Theorem \ref{thm: smooth}, let $\mu$ denote the homogeneous maximal ideal of $S$ and $\nu$ the ideal of $T$ generated by the $z^{\ell}_{ij}$ for which $x^\ell_{ij} \in Y_1 \setminus \{x^{\ell'}_{i'j'}\}$.  Then \[
\dim(\mathcal{X}) = \dim(S_\mu)-1 = \dim(S_\mu[1/x^{\ell'}_{i'j'}]) = \dim(T_\nu) = |Y_1|-1 = (m-1)+(n-1)+(r-1),
\] where the $m-1$ counts the $x^{\ell}_{ij} \neq x^{\ell'}_{i'j'}$ with $\ell = \ell'$ and $j=j'$, the $n-1$ counts the $x^{\ell}_{ij} \neq x^{\ell'}_{i'j'}$ with $\ell = \ell'$ and $i=i'$, and the $r-1$ counts the $x^{\ell}_{ij} \neq x^{\ell'}_{i'j'}$ with $i = i'$ and $j=j'$.
\end{proof}

Although this dimension count appears in \cite[Corollary 4.5]{FK}, we note that the computation and proof are much simpler when considering the toric case alone.

To conclude this section, we will show that the toric case is somewhat broader than it might first appear.  In the $r=2$ case, all double determinantal varieties in which the minors taken in either $H$ or in $V$ are of size 2 are either ordinary determinantal varieties or toric.  The precise statement follows:

\begin{theorem}
Let $I = I_a(H)+I_b(V)$ be a double determinantal ideal with $H$ and $V$ determined by $2$ matrices of size $m \times n$ with $m, n \geq 2$.  If $s = 2$ or $t=2$, and $s+t > 4$, then $I$ is the defining ideal of a determinantal variety.  
\end{theorem}

\begin{proof}
For convenience, use the notation $X = X_1$ and $Y = Y_1$ for the two matrices whose horizontal and vertical concatenations are $H$ and $V$, respectively.  Suppose $a = 2$ and $b>2$. We will show that $I = I_2(H)$ by showing that $I_b(V) \subseteq I_2(H)$.  Fix a $b$-minor $\delta$ of $V$.  Then $\delta$ is determined by a choice of $b$ rows and $b$ columns of $V$.  Because $b>2$, by the pigeon hole principle, at least one of the matrices $X$ or $Y$ contains at least two of the rows determining $\delta$.  If at least two of the rows determining $\delta$ are rows of $X$, fix two such rows and form the submatrix $Z$ of $X$ whose rows are exactly those two rows and whose columns are the $b$ columns determining $\delta$  Let $W$ be the submatrix of $V$ whose rows are the $b-2$ rows of $V$ determining $\delta$ other than the $2$ involved in $Z$ and whose columns are the $b$ columns determining $\delta$.  Observe that $I_2(Z) \subseteq I_2(X) \subseteq I_2(H)$.  Then $\delta = \sum_{\substack{\gamma \in I_2(Z)\\ \varepsilon \in I_{b-2}(W)}} \gamma \cdot \varepsilon \in I_2(H)$ because each summand of our expression of $\delta$ has a factor that is an element of $I_2(H)$.  

The constructions are similar to the case above if we instead fix two rows determining $\delta$ that are rows of $Y$ and in the cases for $a>2$ and $b=2$.
\end{proof}

\section{Empirical Testing}\label{sect: testing}

We conclude by describing a conjecture in commutative algebra for which toric double determnantal ideals are a natural testing ground.  We provide some preliminary steps in conducting such tests.  For a local ring $T$ of dimension $d$ with maximal ideal $\mu$ and $\mu$-primary ideal $J$ of $T$, let $e(J)$ denote the Hilbert--Samuel multiplicity of $J$, and let $\ell(T/J)$ denote the length of $T/J$, i.e. the $T/\mu$-vector space dimension of $T/J$, which is finite.  For background on the Hilbert--Samuel multiplicity, we refer the reader to \cite[Chapter 12]{Eis95}.  Lech's inequality states that $e(J) \leq d! e(m) \ell(T/J)$ \cite[Theorem 3]{Lec60} and has seen various improvements and extensions since its original appearance (see, for example, \cite{Han02, HSV17, KMQ+20}).  Recently, it was conjectured that an asymptotic version of Lech's inequality holds when $T_P$ is regular for every prime $P \neq \mu$.  The precise statement of the conjecture is that, in that case, $\displaystyle \lim_{N \rightarrow \infty} $$\sup_{\substack {\sqrt{J} = \mu  \\ \ell(T/J)>N}} \left\{\frac{e(J)}{d! \ell(T/J)} \right\} = 1$ \cite[Conjecture 1.2 (a)]{HMQS}.  Considering the careful choice of powers of $\mu$ as the ideals $J$ shows that this limit must be at least $1$, and so the question is whether for every $\varepsilon>0$, there exists $N \gg 0$ such that for any $\mu$-primary ideal $J$ with $\ell(T/J) > N$, we have $e(J) \leq d!(1 + \varepsilon)\ell(T/J)$. Otherwise, and less precisely, stated, this conjecture asserts no $\mu$-primary ideals give asymptotically larger values of $\frac{e(J)}{d! \ell(T/J)}$ than powers of $\mu$, that the powers of $\mu$, in the limit, serve as a weak bound for this ratio.

Because the localizations of the homogeneous coordinate rings of toric double deteminantal varieties at their non-maximal homogeneous prime ideals are regular by Theorem \ref{thm: smooth}, it follows from homogeneity that they satisfy the full hypothesis of this conjecture and so that they are an appropriate example to use to explore it.  We report results consistent with the truth of the conjecture in the smallest nontrivial example of a toric double determinantal variety.  Because the conjecture is known over fields of prime characteristic when $T/\mu$ is perfect \cite[Corollary 4.4]{HMQS}, our computations, done in Macaulay2 \cite{M2} are over the rationals.  

Because these computations are both challenging theoretically and also computationally costly, we considered the smallest nontrivial example: the case of 2 matrices $X$ and $Y$, each of dimension $2 \times 2$.  With notation as in Notation \ref{notation}, we generated $175$ ideals $J$ of $R$ using the $x_{ij}^{12}$ and $y_{ij}^{12}$ for all $1 \leq i,j \leq 2$ together with between $1$ and $6$ additional monomial generators each of degree at most $10$ produced by the \texttt{randomMonomialIdeal} command in the  \texttt{RandomIdeals} package.  Because every $\mu$-primary ideal contains a power of each of the variables, it is not artificial to include some power of each $x_{ij}$ and $y_{ij}$.  The reason for choosing that power to be $12$ and for choosing the other generators of $J$ to be monomials in $R$ is to control the timing required to compute $e(JS)$ (where the multiplicity is understood to be over $S$).  Scatterplots showing ordered pairs $(\ell(S/JS), e(JS)/d!\ell(S/JS))$ (Figure 1) and $($number of random monomial generators, $e(JS)/d!\ell(S/JS))$ (Figure 2) appear below.  In both cases, we report decimal truncations of the ratios $e(JS)/d!\ell(S/JS))$.  Conjecture \cite[Conjecture 1.2 (a)]{HMQS} sounds plausible in light of these data.  Note that the Hilbert--Samuel multiplicity of $S$ is $6$ and that $\dim(S) = 4$.  With $\mu = (x_{11}, \ldots ,y_{22})$, the homogeneous maximal ideal of $S$, we provide for scale a table of values $\ell(S/\mu^k)$ and $\frac{e(\mu^k)}{d! \ell(S/\mu^k)}$ for $1 \leq k \leq 12$.  

\begin{center}
\begin{tabular}{ |c|c|c|c|c|c|c|c|c|c|c|c|c| } 
 \hline
$k$ & 1 & 2 & 3& 4 & 5&6&7&8&9&10&11 &12\\ 
 \hline
 $\ell(S/\mu^k)$ &1& 9& 36& 100& 225& 441& 784& 1296& 2025& 3025& 4356& 6084 \\ 
 \hline
$\frac{e(\mu^k)}{4! \ell(S/\mu^k)}$ & .25& .444& .562& .64& .694& .734& .766& .79& .81& .826& .84& .853 \\ 
 \hline
\end{tabular}
\end{center}

\bigskip

Data on the ideals described in the paragraph above appear below.  Figure 1 shows that, although the ratio $e(JS)/d!\ell(S/JS)$ may tend to grow with $\ell(S/JS)$, those ratios remain much smaller than the expected asymptotic bound given by the powers $\mu^k$ of $\mu$ for similar $\ell(S/JS)$ and $\ell(S/\mu^k)$.  Figure 2 shows that that is variability among the $e(JS)/d!\ell(S/JS)$ for a fixed number of random monomial generators.

\bigskip

\hspace{-0.5cm}
\begin{minipage}{2in}
\begin{tikzpicture}
\begin{axis}[
    title = {Figure 1},
    xlabel={$\ell(S/JS)$ },
    ylabel={$\dfrac{e(JS)}{4!\ell(S/JS)}$},
    xmin=0, xmax=40000,
    ymin=0, ymax=.2,
    xtick={0,10000,20000,30000,40000},
    ytick={0,.1,.2},
    legend pos=north west,
    ymajorgrids=true,
    grid style=dashed,
]

\addplot[
    color=black,
    mark=x,
    only marks,
    scatter,
    colormap/blackwhite
    ]
    coordinates {
    (33238,.104)(6480,.0667)(12168,.071)(33238,.104)(6480,.0667)(12168,.071)(36780,.117)(25129,.0859)(33200,.104)(6480,.0667)(17128,.0757)(12167,.071)(6480,.0667)(6480,.0667)(25170,.0858)(21428,.0807)(31029,.0974)(21442,.0806)(28345,.0914)(21434,.0806)(25157,.0858)(17128,.0757)(33269,.104)(17129,.0757)(30961,.0977)(33696,.103)(12167,.071)(12168,.071)(27640,.101)(5151,.0652)(10233,.0703)(21380,.0892)(4308,.0613)(6049,.0675)(5835,.0658)(16758,.0759)(4944,.0656)(28310,.103)(28920,.103)(16160,.078)(22739,.087)(9517,.0706)(5154,.0652)(21588,.0884)(19314,.0808)(17158,.0825)(29636,.0996)(7206,.0699)(22697,.0936)(4941,.0656)(10920,.0703)(17111,.0828)(15705,.0825)(11271,.0773)(6009,.0692)(9869,.0706)(22848,.105)(9537,.0718)(14377,.0793)(11266,.0767)(5812,.0637)(17363,.0857)(13611,.0793)(20014,.0863)(26337,.109)(13211,.0771)(5479,.0679)(4482,.0642)(10831,.0709)(11099,.0712)(4205,.0627)(4049,.0649)(3614,.0631)(5188,.0662)(3576,.0632)(3897,.0647)(12044,.0757)(9840,.0707)(17450,.101)(19407,.101)(1141,.0537)(2968,.0616)(18190,.0856)(13073,.0862)(10511,.0864)(24803,.104)(9834,.0808)(6901,.0691)(601,.0466)(12785,.0798)(3742,.0665)(9395,.0728)(2897,.0649)(22125,.105)(18413,.0978)(8970,.0776)(8290,.0728)(14425,.0848)(12905,.0818)(3391,.0602)(3097,.0657)(1420,.0634)(4234,.0652)(7110,.0708)(19922,.102)(11463,.0797)(9865,.0765)(14563,.0989)(11906,.0829)(3002,.0599)(6700,.0677)(5441,.0717)(7362,.0793)(380,.0523)(4448,.0751)(2421,.0612)(3635,.0671)(2594,.064)(15022,.0876)(1025,.0539)(6834,.0747)(3287,.0656)(14601,.0933)(5667,.0693)(4366,.0688)(3654,.067)(11989,.0782)(3825,.0681)(15956,.0938)(332,.0474)(1257,.0542)(8512,.0834)(17822,.0909)(4451,.0708)(9146,.0801)(2123,.0599)(7300,.084)(2517,.0624)(6747,.0953)(2331,.0601)(6018,.0745)(3430,.0654)(11584,.087)(2300,.0687)(2091,.064)(2217,.0596)(3306,.0684)(3105,.0671)(2522,.0671)(2996,.0673)(4524,.0671)(9062,.0792)(550,.0563)(3500,.0691)(433,.0497)(3333,.0621)(3581,.067)(7804, .0819)(11917,.0826)(25970,.123)(1373,.0618)
    (10758,.0736)(5256,.0653)(7663,.0704)(11832,.0822)(10197,.0706)(2464,.0585)(17092,.0929)(20108,.0963)(12777,.078)(9666,.072)(17944,.0836)(10077,.0756)(19059,.0831)(15885,.0778)
    };
    
\end{axis}
\end{tikzpicture}
\end{minipage}

\vspace{-2.545in}
\hspace{7.25cm}
\begin{minipage}{2in}
\begin{tikzpicture}
\begin{axis}[
    title = {Figure 2},
    xlabel={Number of random monomials},
    ylabel={$\dfrac{e(JS)}{4!\ell(S/JS)}$},
    xmin=0, xmax=7,
    ymin=0, ymax=.2,
    xtick={0,1,2,3,4,5,6,7},
    ytick={0,.1,.2},
    legend pos=north west,
    ymajorgrids=true,
    grid style=dashed,
]

\addplot[
    color=black,
    mark=x,
    only marks,
    scatter,
    colormap/blackwhite
    ]
    coordinates {
    (1,.104)(1,.0667)(1,.071)(1,.104)(1,.0667)(1,.071)(1,.117)(1,.0859)(1,.104)(1,.0667)(1,.0757)(1,.071)(1,.0667)(1,.0667)(1,.0858)(1,.0807)(1,.0974)(1,.0806)(1,.0914)(1,.0806)(1,.0858)(1,.0757)(1,.104)(1,.0757)(1,.0977)(1,.103)(1,.071)(1,.071)
    (2,.101)(2,.0652)(2,.0703)(2,.0892)(2,.0613)(2,.0675)(2,.0658)(2,.0759)(2,.0656)(2,.103)(2,.103)(2,.078)(2,.087)(2,.0706)(2,.0652)(2,.0884)(2,.0808)(2,.0825)(2,.0996)(2,.0699)(2,.0936)(2,.0656)(2,.0703)(2,.0828)(2,.0825)
    (3,.0773)(3,.0692)(3,.0706)(3,.105)(3,.0718)(3,.0793)(3,.0767)(3,.0637)(3,.0857)(3,.0793)(3,.0863)(3,.109)(3,.0771)(3,.0679)(3,.0642)(3,.0709)(3,.0712)(3,.0627)(3,.0649)(3,.0631)(3,.0662)(3,.0632)(3,.0647)(3,.0757)(3,.0707)
    (4,.101)(4,.101)(4,.0537)(4,.0616)(4,.0856)(4,.0862)(4,.0864)(4,.104)(4,.0808)(4,.0691)(4,.0466)(4,.0798)(4,.0665)(4,.0728)(4,.0649)(4,.105)(4,.0978)(4,.0776)(4,.0728)(4,.0848)(4,.0818)(4,.0602)(4,.0657)(4,.0634)(4,.0652)
    (5,.0708)(5,.102)(5,.0797)(5,.0765)(5,.0989)(5,.0829)(5,.0599)(5,.0677)(5,.0717)(5,.0793)(5,.0523)(5,.0751)(5,.0612)(5,.0671)(5,.064)(5,.0876)(5,.0539)(5,.0747)(5,.0656)(5,.0933)(5,.0693)(5,.0688)(5,.067)(5,.0782)(5,.0681)(5,.0938)(5,.0474)(5,.0542)(5,.0834)(5,.0909)(5,.0708)(5,.0801)(5,.0599)(5,.084)
    (6,.0624)(6,.0953)(6,.0601)(6,.0745)(6,.0654)(6,.087)(6,.0687)(6,.064)(6,.0596)(6,.0684)(6,.0671)(6,.0671)(6,.0673)(6,.0671)(6,.0792)(6,.0563)(6,.0691)(6,.0497)(6,.0621)(6,.067)(6, .0819)(6,.0826)(6,.123)(6,.0618)
    (3,.0736)(3,.0653)(3,.0704)(3,.0822)(3,.0706)(3,.0585)(3,.0929)(3,.0963)(3,.078)(3,.072)(3,.0836)(3,.0756)(3,.0831)(3,.0778)
    };
    
\end{axis}
\end{tikzpicture}
\end{minipage}

\end{document}